\documentclass[12pt,leqno,twoside]{amsart}

\usepackage[margin=1.25in]{geometry}

\usepackage[normalem]{ulem}
\usepackage{amsmath,amscd,amssymb,amsfonts,latexsym,wasysym, mathrsfs, mathtools,hhline,color, tikz}

\usetikzlibrary{calc}

\usepackage[all, cmtip]{xy}
\usepackage{csquotes}
\usepackage{url}
\usepackage{comment}
\usepackage[utf8]{inputenc}

\definecolor{hot}{RGB}{65,105,225}

\usepackage[pagebackref=true,colorlinks=true, linkcolor=hot ,  citecolor=hot, urlcolor=hot]{hyperref}
\usepackage{ textcomp }
\usepackage{ tipa }
\usepackage{graphicx,enumerate}

\usepackage{geometry}
\usepackage{chngcntr}

\theoremstyle{plain}
\newtheorem{theorem}{Theorem}[section]
\newtheorem{prop}[theorem]{Proposition}

\newtheorem{lm}[theorem]{Lemma}

\newtheorem{cor}[theorem]{Corollary}

\newtheorem{lemma}[theorem]{Lemma}
\newtheorem{thrm}[theorem]{Theorem}

\theoremstyle{definition}

\newtheorem{defn}[theorem]{Definition}

\newtheorem{rmk}[theorem]{Remark}

\newtheorem{ex}[theorem]{Example}
\newtheorem*{ex*}{Example}

\def\be{\begin{equation}}
\def\ee{\end{equation}}

\def\bt{\begin{thrm}}
\def\et{\end{thrm}}

\def\bc{\begin{cor}}
\def\ec{\end{cor}}

\def\br{\begin{rmk}}
\def\er{\end{rmk}}

\def\bp{\begin{prop}}
\def\ep{\end{prop}}

\def\bl{\begin{lm}}
\def\el{\end{lm}}

\def\bex{\begin{ex}}
\def\eex{\end{ex}}

\def\bd{\begin{defn}}
\def\ed{\end{defn}}

\newcommand{\CP}{\mathbb{CP}}

\newcommand{\C}{\mathbb{C}}
\newcommand{\Z}{\mathbb{Z}}

\newcommand{\K}{\mathbb{K}}

\newcommand{\sA}{\mathcal{A}}
\newcommand{\sB}{\mathcal{B}}
\newcommand{\cX}{\mathcal{X}}
\newcommand{\cE}{\mathcal{E}}

\newcommand{\sL}{\mathcal{L}}

\newcommand{\sP}{\mathcal{P}}
\newcommand{\sF}{\mathcal{F}}

\title[]{Maximal twisted Betti numbers of complex hyperplane arrangement complements}

\author{Yongqiang Liu}
\address{Y. Liu: The Institute of Geometry and Physics, University of Science and Technology of China, 96 Jinzhai Road, Hefei 230026 P.R. China} 
\email{liuyq@ustc.edu.cn}
\author{Laurentiu Maxim}
\address{L. Maxim: Department of Mathematics,  University of Wisconsin-Madison,  480 Lincoln Drive, Madison WI 53706-1388, USA,  \newline
{\text and} Institute of Mathematics of the Romanian Academy, P.O. Box 1-764, 70700 Bucharest, ROMANIA.}
\email {maxim@math.wisc.edu}
\author{Botong Wang}
\address{B. Wang: Department of Mathematics,         University of Wisconsin-Madison,  480 Lincoln Drive, Madison WI 53706-1388, USA.}
\email {wang@math.wisc.edu}

\date{\today}



\begin{document}

\begin{abstract}
We show that the Betti numbers of a local system on the complement of an essential complex hyperplane arrangement are maximized precisely when the local system is constant. This result answers positively a recent question of Yoshinaga and the first author.
\end{abstract}

\maketitle

\section{Introduction}

Let $\sA=\{H_1,\ldots, H_d\}$ be an essential affine hyperplane arrangement in $\C^n$, with complement $U_\sA$. As shown independently by Dimca-Papadima \cite{DP} and Randell \cite{R}, the complement  $U_\sA$ has the homotopy type of a minimal CW complex, hence, in particular, its cohomology groups $H^i(U_\sA;\Z)$ are free abelian, for any integer $i \geq 0$. The Betti numbers $b_i(U_\sA)$ of $U_\sA$ are classically known, e.g., see \cite{OT}.

Let us now fix a coefficient field $\K$  and consider a $\K$-local system $\sL$ on $U_\sA$ of rank $r$. For $i \geq 0$ an integer, denote by 
\[
b_i(U_\sA;\sL):=\dim_\K H^i(U_\sA;\sL)
\]
the corresponding twisted Betti numbers of $U_\sA$. Then it follows by definition, together with the existence of a minimal CW structure on $U_\sA$, that (cf. also \cite{Coh98} for the case $\K=\C$)
\begin{equation}\label{ineq}
b_i(U_\sA;\sL) \leq r \cdot c_i(U_\sA)=r \cdot b_i(U_\sA),
\end{equation}
where $c_i$ denotes the number of $i$-cells in a minimal CW structure of $U_\sA$. Note that for CW complexes which do not admit a minimal structure (e.g., an arbitrary hypersurface arrangement complement) it is difficult in general to compare the leftmost and rightmost terms of \eqref{ineq}, as both are dominated by the middle term in \eqref{ineq}.

It is therefore natural to investigate for which local systems the above inequality \eqref{ineq}  becomes an equality for some $0 \leq i \leq n$. For rank-one local systems, this question was answered in full by Liu-Yoshinaga in the case when $\sA$ is a complexified real arrangement, see \cite[Theorem 1.1]{LY}.

The main result of this paper settles the question for all complex arrangements, in particular also providing a positive answer to Question 1.3 in \cite{LY}. We prove the following.
\begin{theorem}\label{thmain}
Let $\sA$ be an essential affine hyperplane arrangement in $\C^n$, with complement $U_\sA$. Let $\K$ be a field, and $\sL$ a nontrivial rank $r$ $\K$-local system on $U_\sA$ with twisted Betti numbers $b_i(U_\sA;\sL)$.  Then for any $0\leq i \leq n$ we have
\[
b_i(U_\sA;\sL)<r\cdot b_i(U_\sA).
\]
\end{theorem}

The following consequence is immediate.
\begin{cor}
If the inequality \eqref{ineq}  is an equality for some $0 \leq i \leq n$, then it is an equality for all $i$ in this range, and this can only happen for the constant sheaf.
\end{cor}

The proof of Theorem \ref{thmain} is by induction on the dimension $n$ of the ambient space, by first reducing to the case of central arrangements.

As a byproduct of the proof, we also get the following result.
\begin{prop} \label{prop surj} Let $\sA$ be an essential affine hyperplane arrangement in $\C^n$ with complement $U_\sA$. For a point $x$ in one of the hyperplanes of $\sA$, let $B_x\subset \C^n$ be a sufficiently small ball centered at $x$. Let $\sL$ be any finite rank $\K$-local system on $U_{\sA}$. 
Then the pullback homomorphism
\[
H^{n}(U_\sA; \sL)\to H^{n}(U_\sA\cap B_x ; \sL)
\]
is surjective. Moreover, we have the following surjective map 
\begin{equation} \label{epi 2}
  H^{n}(U_\sA; \sL)\twoheadrightarrow \bigoplus_{x\in L_0(\sA)} H^{n}(U_\sA\cap B_x ; \sL)  
\end{equation}
with $L_0(\sA)$ denoting the $0$-dimensional flats (or edges) of $\sA$.
\end{prop}
\begin{rmk}
    When $\sL$ is the constant sheaf, (\ref{epi 2}) becomes an isomorphism, known as the Brieskorn decomposition. So (\ref{epi 2}) should be viewed as a generalization of Brieskorn decomposition to arbitrary local systems. 
\end{rmk}

\medskip

{\bf Acknowledgments.} This question was asked during the Workshop and Summer School on Hyperplane Arrangements, held at Tongji University in Shanghai, China, during Summer 2025. We thank the local organizers, Xiping Zhang and Fangzhou Jin, for providing us with a stimulating research environment and for excellent working conditions. We also thank J. Sch\"urmann and A. Parusi\'nski for useful discussions.
Y. Liu is supported by the Project of Stable Support for Youth Team in Basic Research Field, CAS (YSBR-001), NSFC grant No. 12571047 and  the starting grant from University of Science and Technology of China. 
L. Maxim 
acknowledges support from the Simons Foundation and from the project ``Singularities and Applications'' - CF 132/31.07.2023 funded by the European Union - NextGenerationEU - through Romania's National Recovery and Resilience Plan. 

\section{Preliminaries}
In this section we develop the tools necessary for proving our main result Theorem \ref{thmain} (see also the Appendix). For the constructible sheaf calculus, we use the notations from \cite{Dim04,MS}. 

\begin{lemma} \label{lemma 1}
Let $\sP$ be a $\K$-perverse sheaf on $\C^{n}$, and let $x$ be any point on $\C^{n}$ with inclusion map  $i_x: \{x\}\hookrightarrow \C^{n}$. Then the natural homomorphism 
\[
H^0(\C^{n}; \sP)\to H^0(x; i_x^{-1}\sP)
\]
is surjective. 
\end{lemma}
\begin{proof}
Let $j_x: \C^{n}\setminus \{x\}\hookrightarrow \C^{n}$ be the open embedding. Then we have a distinguished triangle
\[
j_{x!}j_x^{-1}\sP\to \sP\to i_{x*}i_x^{-1}\sP\xrightarrow{+1} .
\]
Since $j_x^{-1}$ is $t$-exact and $j_{x!}$ is right $t$-exact (e.g., see \cite[Theorem 5.2.4]{Dim04}), it follows that $j_{x!}j_x^{-1}\sP\in \,^p D_c^{\leq 0}(\C^{n},\K)$. Hence, by Artin vanishing (e.g., see \cite[Theorem 10.3.59]{MS}), we have that
\[
H^i(\C^{n}; j_{x!}j_x^{-1}\sP)=0
\]
for all $i>0$. Then by the long exact sequence associated to the above triangle, it follows that the homomorphism
\[
H^0(\C^{n}; \sP)\to H^0(\C^{n}; i_{x*}i_x^{-1}\sP)\cong H^0(x; i_x^{-1}\sP)
\]
is surjective. 
\end{proof}

 Lemma \ref{lemma 1} directly implies Proposition \ref{prop surj}.
\begin{proof}[Proof of Proposition \ref{prop surj}]
Since the inclusion map $j\colon U_\sA \hookrightarrow \C^n$ is a quasi-finite affine morphism, it follows that $Rj_* \sL[n]$ is a $\K$-perverse sheaf on $\C^n$ (e.g., see \cite[Corollary~5.2.17]{Dim04}).
  The first assertion in the statement then follows by applying Lemma \ref{lemma 1} to $Rj_* \sL[n]$. The second claim follows by applying Lemma \ref{lemma 1} to $Rj_* \sL[n]$ upon replacing $x$ by the finite set of points $L_0(\sA)$ in $\C^n$.
\end{proof}

On the other hand, Lemma \ref{lemma 1}  can be used to prove the following result, which is crucial for the proof of our main Theorem \ref{thmain}.

\begin{prop}\label{prop_surj} Let $\sA$ be an affine hyperplane arrangement in $\C^n$, and let $U_\sA=\C^n \setminus \sA$ be its complement. Let $\sL$ be any finite rank $\K$-local system on $U_{\sA}$. Let $x$ be a point on one of the hyperplanes of $\sA$, and let $B_x\subset \C^n$ be a sufficiently small ball centered at $x$. Let $f\colon \C^n\to \C$ be a general linear function, and let $c\in \C$ be a point different but sufficiently close to $f(x)$ (relative to the radius of $B_x$). 
Then the pullback homomorphism
\[
H^{n-1}(U_\sA\cap f^{-1}(c); \sL)\to H^{n-1}(U_\sA\cap B_x \cap  f^{-1}(c); \sL)
\]
is surjective. 
\end{prop}

\begin{proof}
Let $j\colon U_\sA \hookrightarrow  \C^n$ be the open embedding, which is a quasi-finite affine morphism. Then, as already explained in the proof of Proposition \ref{prop surj}, the complex $\sF=Rj_*\sL[n]$ is a perverse sheaf on $\C^n$. Hence, the perverse nearby cycle complex $${^p}\psi_{f-f(x)}(\sF)=\psi_{f-f(x)}(\sF)[-1]$$ associated to $\sF$ is a perverse sheaf on $f^{-1}(f(x))$. 

Let us first note that by the construction of nearby cycles, we have natural isomorphisms 
\begin{equation}\label{id1}
H^{n-1}(U_\sA\cap B_x \cap  f^{-1}(c); \sL) \cong H^{-1}(B_x \cap f^{-1}(c); \sF) \cong H^0(x; i_x^{-1}{^p}\psi_{f-f(x)}(\sF)).
\end{equation}
Furthermore, there are natural isomorphisms
\begin{equation}\label{id2}
H^{n-1}(U_\sA\cap f^{-1}(c); \sL) \cong H^{-1}(f^{-1}(c); \sF) \cong H^{0}(f^{-1}(f(x)); {^p}\psi_{f-f(x)}(\sF)),
\end{equation}
where the first isomorphism in \eqref{id2} is immediate, while the second isomorphism will be proved in Proposition \ref{near} in the Appendix (the main difficulty here being the fact that $f$ is not a proper map).
The desired surjectivity in the statement follows from Lemma \ref{lemma 1}, by using \eqref{id1} and \eqref{id2}.
\end{proof}

The following result was proved by Cohen for rank one local systems (see \cite{Coh98}, and also \cite{Coh93}). 
\begin{prop}  \label{prop} Let $H$ be a generic hyperplane in $\C^n$. For $\sL$ a rank $r$ $\K$-local system on $U_\sA$, we have
    $$\dim_\K H^i (U_\sA, U_\sA\cap H; \sL)=\begin{cases}
       r\cdot b_n(U_\sA), & {\rm if } \ i=n,\\
       0, & {\rm otherwise}.
    \end{cases} $$
\end{prop}
\begin{proof}
Since $H$ is generic, 
by the Zariski Theorem of Lefschetz type (e.g., see \cite[Theorem~1.6.5]{Dim92}),
the inclusion map $U_\sA \cap H \hookrightarrow U_\sA $ is an $(n-1)$-equivalence.
Then by \cite[Lemma 5.6]{PS19}, 
the pullback homomorphism $ H^i(U_\sA; \sL)\to  H^i(U_\sA\cap H; \sL)$ is an isomorphism for $0\leq i\leq n-2 $ and is injective for $i=n-1$. 
By the long exact sequence of the cohomology of the pair $(U_\sA, U_\sA \cap H)$,
\[
\cdots \to H^i(U_\sA, U_\sA\cap H; \sL)\to H^i(U_\sA; \sL)\to H^i(U_\sA\cap H; \sL)\to \cdots
\]
we get that $ H^i (U_\sA, U_\sA\cap H; \sL)=0$ for $i\leq n-1$. On the other hand, since $U_\sA$ (resp., $U_\sA \cap H$) is affine,  it is homotopy equivalent to an $n$ (resp., $n-1$) dimensional CW complex. In particular, $H^i(U_\sA;\sL)=0$ for $i>n$ and $H^i(U_\sA\cap H;\sL)=0$ for $i>n-1$. By the above relative long exact sequence, we get that $ H^i (U_\sA, U_\sA\cap H; \sL)=0$ for $i>n$.
Hence $H^i(U_\sA, U_\sA\cap H;\sL)=0$ for $i\neq n$ and
\[
\begin{split}
\dim_\K H^{n}(U_\sA, U_\sA\cap H; \sL)& = \vert \chi(U_\sA, U_\sA\cap H; \sL) \vert =\vert \chi(U_\sA; \sL)-\chi(U_\sA\cap H; \sL) \vert \\ &=
r\cdot \vert \chi(U_\sA)-\chi(U_\sA\cap H) \vert=r\cdot \vert \chi(U_\sA\setminus H) \vert\\
&=r\cdot b_n(U_\sA),
\end{split}
\]
where the last equality is proved in \cite[Lemma 5]{DP}.
\end{proof}

\begin{cor} 
Let $H$ be a generic hyperplane in $\C^n$. For a rank $r$ $\K$-local system $\sL$ on $U_\sA$, if 
  $ \dim_\K H^n(U_\sA; \sL)=r\cdot b_n(U_\sA)$,
    then we have a natural isomorphism
    \begin{equation}\label{eq_44}
H^{n-1}(U_\sA; \sL)\cong H^{n-1}(U_\sA\cap H; \sL).
\end{equation}
\end{cor}
\begin{proof}
  The long exact sequence for the cohomology of the pair $(U_\sA, U_\sA \cap H)$ reduces as in 
Proposition \ref{prop} to a $4$-term exact sequence 
\begin{equation} \label{4 terms}
    0 \to H^{n-1}(U_\sA; \sL)\to H^{n-1}(U_\sA\cap H; \sL)\to  H^{n}(U_\sA, U_\sA\cap H; \sL)\to H^{n}(U_\sA; \sL) \to 0.
\end{equation}
By Proposition \ref{prop}, our assumption on dimension implies that the surjective homomorphism 
\[
H^{n}(U_\sA, U_\sA\cap H; \sL)\to H^{n}(U_\sA; \sL)
\]
is in fact an isomorphism. 
Then the claim follows by using the exact sequence \eqref{4 terms}.
\end{proof}

\section{Proof of the main result}

\begin{proof}[Proof of Theorem \ref{thmain}]
 We first give the proof of the theorem assuming $i=n$, by induction on $n$ in two steps.  If $n=1$, the claim is obvious. For induction, we assume that the claim holds for any essential hyperplane arrangement in $\C^{n-1}$.

\medskip

\noindent Step 1:   Assume that $\sA$ is a central arrangement in $\C^n$. We reduce the proof  to a lower dimension case, then use induction. Assume that the local system $\sL$ is defined by a homomorphism on $\pi_1(U_\sA)$ which sends the meridian corresponding to the hyperplane $H_i\in \sA$  to a matrix $A_i \in \mathrm{GL}_r(\K)$, for each $i=1,\ldots,d$. 

Let $M_\sA$ denote the  complement of the
corresponding projective arrangement in $\mathbb{CP}^{n -1}$, and consider the Hopf map $p\colon U_\sA \to M_\sA$. 
Then we have a spectral sequence 
with $E_2$-tem given by $E_2^{a,b}= H^a(M_\sA; R^b p_* \sL)$, which converges to $H^{a+b}(U_\sA;\sL)$. Since the fiber of the Hopf map $p$ is $\C^*$, we have that $E_2^{a,b}=0$ for $b\neq 0,1$. 
Following \cite[page 210]{Dim04}, we consider the total turn monodromy operator of the local system $\sL$ as an  invertible operator $T(\sL):\K^r \to \K^r$. This operator plays an important role in describing the local systems $R^b p_* \sL$, $b=0,1$.

If $T(\sL)\neq I_r$ is not the identity matrix, it follows from \cite[Proposition 6.4.3]{Dim04} that 
$R^1 p_* \sL$ is a local system of rank $0\leq r'<r$. 
Since $ M_\sA$ is an $(n-1)$-dimensional affine variety, it is homotopy equivalent to an $(n-1)$-dimensional CW complex. It follows that for the above spectral sequence we have that $E_2^{n-b,b}=0 $ unless $b=1$ and $E_2^{a,b}=0$ if $a+b>n$. Therefore,  
 $$E_2^{n-1,1}= E_\infty^{n-1,1}\cong H^n(U_\sA;\sL) .$$
Then we have $$\dim_\K H^n(U_\sA;\sL) = \dim_\K H^{n-1}(M_\sA; R^1 p_* \sL)\leq r'\cdot b_{n-1}(M_\sA)=r'\cdot b_{n}(U_\sA)<r\cdot b_{n}(U_\sA).$$  For the equality $b_{n-1}(M_\sA)=b_{n}(U_\sA)$, we apply the K\"unneth formula to the isomorphism $U_\sA\cong M_\sA \times \C^*$.

On the other hand,  
if $T(\sL)= I_r$ is the identity matrix, 
there exists a local system $\sL'$ of rank $r$ on $M_\sA$ (constructed as in \cite[Proposition 6.4.3]{Dim04}) such that $\sL=p^{-1} \sL'$. Moreover, by the K\"unneth formula we have the following isomorphism (note that $T(\sL)$ corresponds to the image of the representation defining $\sL$ on the factor $\C^*$)
\begin{center}
    $  H^{n-1}(M_\sA;\sL')\cong  H^{n}(U_\sA;\sL) $.
    \end{center}
Since $b_{n-1}(M_\sA)= b_{n}(U_\sA)$, and $M_\sA$ can be seen as a complement 
to an essential hyperplane arrangement in $\C^{n-1}$ (by setting one of the hyperplanes as the hyperplane at infinity in $\mathbb{CP}^{n -1}$),  the claim follows  by induction.

\medskip
 
\noindent Step 2: Next we prove the claim when the arrangement $\sA$ in $\C^n$ is not necessarily central by reducing to the central case with the same dimension for the ambient space.

Assume by contradiction that the inequality \eqref{ineq} is an equality for $i= n$. 
If $\sL$ is not the constant sheaf, without loss of generality, we assume that $A_1 \neq I_r$. Since $\sA$ is essential, there exists an intersection point $x\in L_0(\sA)$ contained in $H_1$. 
Let $B_x$ be a sufficiently small open ball in $\C^n$ centered at $x$. Consider a generic hyperplane $H$ sufficiently close to $x$ (relative to the radius of $B_x$). 
Consider the commutative diagram
\[
\xymatrix{
H^{n-1}(U_\sA; \sL) \ar[d]\ar@{->>}[r]&H^{n-1}(U_\sA\cap H; \sL) \ar@{->>}[d]\\
H^{n-1}(U_\sA\cap B_x; \sL)\ar[r]& H^{n-1}(U_\sA\cap B_x\cap H; \sL).
}
\]
The isomorphism \eqref{eq_44} implies that the top horizontal arrow is surjective. Moreover, Proposition \ref{prop_surj} implies that the  vertical arrow on the right is surjective. 
Therefore, the commutativity of the diagram implies that the bottom horizontal arrow must be surjective. 
Then by the relative long exact sequence, the surjectivity of the bottom horizontal map implies that
\begin{equation}\label{eq_iso1}
H^{n}(U_\sA\cap B_x, U_\sA\cap B_x\cap H; \sL)\to H^{n}(U_\sA\cap B_x; \sL)
\end{equation}
is an isomorphism. 
Since $U_\sA\cap B_x$ can be viewed as a complement of central hyperplane arrangement, by Proposition \ref{prop}, we have $$\dim_\K H^{n}(U_\sA\cap B_x; \sL) =\dim_\K H^{n}(U_\sA\cap B_x, U_\sA\cap B_x\cap H; \sL)=r\cdot b_n(U_\sA\cap B_x).$$ 
By Step 1 applied to the restriction of  $\sL$  over $U_\sA\cap B_x$, this restriction must be the constant sheaf, which contradicts the assumption that $A_1\neq I_r$.

\medskip

We conclude the proof by showing that it can be reduced to the case $i=n$, which was already proved above.
When $i=0$, the claim is obvious. So we assume that $i\geq 1$.  
Set $\sB=\sA \cap L$, where $L$ is a generic $i$-dimensional affine space in $\C^n$. Then $\sB$ is an affine arrangement in $L=\C^i$. By the Lefschetz hyperplane section theorem, we have 
 $\dim_\K H^i(U_\sA,\sL) \leq \dim_\K H^i(U_\sB,\sL)$.
Then we have $$  \dim_\K H^i(U_\sA,\sL) \leq \dim_\K H^i(U_\sB,\sL) < r\cdot b_i(U_\sB)=r\cdot b_i(U_\sA).$$
Here the last equality is  due to the genericity of $L$ and the minimality of CW structures of arrangements, whereas the middle inequality follows from applying the result proved above for $\sB$ as an arrangement in $\C^i$ and the fact that if $\sL$ is not the constant sheaf then also $\sL|_{U_\sB}$ is not constant. 
\end{proof}

\appendix
\section*{Appendix} 
\addcontentsline{toc}{section}{Appendix}

\setcounter{section}{1}
\renewcommand{\thesection}{A}

In this appendix, we include a proof of a result well known to experts (Proposition \ref{propA} and Corollary \ref{cor33}); see, for instance, \cite[Section 6.1]{Tib}. However, since we were unable to find it in the precise form required for this paper, we provide a complete proof here for the reader’s convenience. We also prove Proposition \ref{near}, which is used in the proof of Proposition \ref{prop_surj}.

Let $x_0, \dots, x_n$ be the homogeneous coordinates of $\CP^n$, and let $H_\infty=\{x_0=0\}\subset \CP^n$ be the hyperplane at infinity. Fix a Whitney stratification $\CP^n=\bigsqcup_{i\in I} W_i$ of $\CP^n$, which we denote by $\mathcal{W}$. Let $l=a_0x_0+a_1x_1+\cdots +a_nx_n$ be a linear form, with $a_i\neq 0$ for some $1\leq i\leq n$. Let $\cX$ be the blow-up of $\CP^n$ along the axis $V_l=\{l=x_0=0\}\subset \CP^n$, with exceptional divisor $\cE$. Then taking the ratio $l/x_0$ defines a proper map $\pi: \cX\to \CP^1$. Let $X=\pi^{-1}(\C)$ and $E=\cE\cap X$. Denote the restriction $\pi|_{X}: X\to \C$ by $\pi_X$. In other words, $\pi_X: X\to \C$ is the fiberwise projective compactification of the linear map $l/x_0: \C^n\to \C$, where $\C^n=\CP^n\setminus H_\infty$. We denote the composition $X\to \cX\to \CP^n$ of the blowup map and inclusion maps by $p$. 
\begin{prop}\label{propA}
Under the above notations, for a general linear form $l$, the following statements hold.
\begin{enumerate}
\item The pull-back of the Whitney stratification $\mathcal{W}$ of $\CP^n$ by $p$ defines a Whitney stratification of $X$, which we denote by $\mathcal{W}'$. 
\item The map $\pi_X\colon X\to \C$ does not have critical points along $E$, with respect to the stratification $\mathcal{W}'$. 
\end{enumerate}
\end{prop}
\begin{proof}
Notice that $[x_0:l]$ defines a rational map $\CP^n\dashrightarrow \CP^1$ and $\cX$ is equal to the closure of its graph in $\CP^n\times \CP^1$. Moreover, $X$ is equal to the intersection of $\cX$ with $\CP^n\times \C$. 

Taking cartesian product with $\C$, the Whitney stratification $\mathcal{W}$ of $\CP^n$ induces a Whitney stratification $\mathcal{W}_{\CP^n\times \C}$ on $\CP^n\times \C$. For the first statement, we will show that $X$ intersects every stratum of $\mathcal{W}_{\CP^n\times \C}$ transversally, which implies that the restriction of $\mathcal{W}_{\CP^n\times \C}$ to $X$ is a Whitney stratification. 

Let $t$ be the coordinate of $\C$. Then $X \subset \CP^n\times \C$ is defined by the equation $l=t \cdot x_0$, or
\[
a_0x_0+a_1x_1+\cdots +a_nx_n-t x_0=0
\]
for general $a_0, \dots, a_n$. Let $\mathcal{O}_{\CP^n\times \C}(1)$ be the pullback of the line bundle $\mathcal{O}_{\CP^n}(1)$ by the projection. Then, via pullback, $x_0, \dots, x_n$ define global sections of $\mathcal{O}_{\CP^n\times \C}(1)$. Also via pullback, $t$ is a global function of $\CP^n\times \C$, and hence the product $t x_0$ is also a global section of $\mathcal{O}_{\CP^n\times \C}(1)$. Notice that the global sections $x_0, \dots, x_n, tx_0$ define a base-point free linear system of $\mathcal{O}_{\CP^n\times \C}(1)$, and $X$ is  a general member of this linear system. Therefore, by the Bertini theorem for basepoint-free linear systems, the zero locus of a general section is transversal to the Whitney stratification $\mathcal{W}_{\CP^n\times \C}$. Hence the first statement follows. 

For the second statement, we need to show that for any $x\in \C$, $(\pi_X)^{-1}(x)$ intersects every stratum of $\mathcal{W}'$ transversally along $E$. Notice that 
\[
E= V_l\times \C\subset \CP^n\times \C.
\]
Let $\tilde{x}$ be any point in $(\pi_X)^{-1}(x)\cap E$. Then $\tilde{x}$ is contained in the line $p(\tilde{x})\times \C$ in $E$. Since $\mathcal{W}'$ is defined as the pullback of a stratification of $\CP^n$, the line $p(\tilde{x})\times \C$ is contained in one stratum of $\mathcal{W}'$. Since the line $p(\tilde{x})\times \C$ intersects $(\pi_X)^{-1}(x)$ transversally, the stratum of $\mathcal{W}'$ containing $\tilde{x}$ also intersects $(\pi_X)^{-1}(x)$ transversally. Therefore, the intersection of  $(\pi_X)^{-1}(x)$ and any stratum of $\mathcal{W}'$ is transversal along $E$. Thus, the second statement follows. 
\end{proof}
\begin{rmk}
It is essential to work with $\pi_X$ instead of $\pi$. In fact, if we replace $X, \pi_X$ and $\mathcal{W}'$ by $\cX, \pi$ and the pullback of $\mathcal{W}$ by the blow-up $\cX\to \CP^n$, then the first statement may fail. In fact, let $u, v$ be the homogeneous coordinate of $\CP^1$. If we apply the same arguments, then $\cX$ is a general member of the linear system defined by sections $ux_0, ux_1, \dots, ux_n, vx_0$ of the line bundle $\mathcal{O}_{\CP^n}(1)\boxtimes \mathcal{O}_{\CP^1}(1)$, which has base locus $\{u=x_0=0\}\subset \CP^n\times \CP^1$. When restricted to $\C=\{u\neq 0\}\subset \CP^1$, the base locus is removed. 
\end{rmk}

Let $\sF$ be an algebraic constructible complex of $\K$-vector spaces on $\C^n$. Let $f\colon \C^n\to \C$ be a general affine  function, and let $\widetilde{f}\colon X\to \C$ be the 
fiberwise compactification of $f$. Denote by $k:\C^n \hookrightarrow X$ the natural embedding.
\begin{cor}\label{cor33}
For any $c_0\in \C$, the supports of ${^p\varphi}_{\widetilde{f}-c_0}(Rk_*(\sF))$ and ${^p\varphi}_{\widetilde{f}-c_0}(k_!(\sF))$ consist of finitely many points, all of which are contained in $k(\C^n)\cong \C^n$. 
\end{cor}
\begin{proof}
First, notice that if $l$ is the homogenization of $f$, then $X$ is the same variety as defined at the beginning of the appendix, and $\widetilde{f}=\pi_X$. 
Considering $\CP^n=\C^n\cup H_\infty$, let $\mathcal{W}$ be a Whitney stratification of $\CP^n$ such that $\C^n$ is a union of strata and $\sF$ is constructible with respect to the induced stratification of $\C^n$. Let $\mathcal{W}'$ be defined as in Proposition \ref{propA}. Then $\mathcal{W}'$ is a Whitney stratification of $X$ with respect to which $Rk_*(\sF)$ is constructible. Since the vanishing cycle complex ${^p\varphi}_{\widetilde{f}-c_0}(Rk_*(\sF))$ has support contained in the stratified critical locus of $\widetilde{f}$, using the second statement of Proposition \ref{propA} it follows that this support is contained in $k(\C^n)$, i.e., it coincides with the support of ${^p}\varphi_{f-c_0}(\sF)$. On the other hand, since $f$ is general,  the support of ${^p}\varphi_{f-c_0}(\sF)$ consists of finitely many points in $k(\C^n)\cong \C^n$, thus proving the assertion for the vanishing cycles of $Rk_*(\sF)$. 
The claim about ${^p\varphi}_{\widetilde{f}-c_0}(k_!(\sF))$ follows similarly, or by using Verdier duality.
\end{proof}

We next prove the following result, which is used in Proposition \ref{prop_surj}.

\begin{prop}\label{near} Let $f: \C^n\to \C$ be a general linear function, and let $c_0$ be a fixed point in $\C$. Then for any $\K$-constructible complex $\sF \in D^b_c(\C^n)$ there is a natural isomorphism
\begin{equation}\label{id2n}
H^{-1}(f^{-1}(c); \sF\vert_{f^{-1}(c)}) \cong H^{0}(f^{-1}(c_0); {^p}\psi_{f-c_0}(\sF)),
\end{equation}
for $c\in \C$ sufficiently close to $c_0$. 
\end{prop}

\begin{proof}
The isomorphism in \eqref{id2n} would be true if $f$ were a proper map. Since this is not the case, we need to work with a proper extension $\widetilde{f}$ of $f$, and use the fact proved above that $f$ does not have ``singularities at infinity''. It suffices  to assume that $c_0$ is a stratified critical value of $f$,
the statement being obvious otherwise.

Let $X\subset \C \mathbb{P}^n \times \C$ be as before the standard partial compactification of the graph of $f$, with projection $\widetilde{f}\colon X \to \C$.  Then $\widetilde{f}$ is a proper extension of $f$, whose fibers are the projective closures of the fibers of $f$. Let $k:\C^n \hookrightarrow X$ be the inclusion map. Consider the cartesian diagram
\[
\xymatrix{
f^{-1}(c) \ar[d]_{k_c}\ar[r]^{i_c}&\C^n \ar[d]^{k}\\
\widetilde{f}^{-1}(c) \ar[r]_{{\widetilde i}_c}& X.
}
\]
Since $c$ is a regular value of $f$, there is a base change isomorphism (e.g., see  \cite[Proposition 4.3.1, Remark 4.3.6]{Sch})
\begin{equation}\label{bc}
{{\widetilde i}_c}^{-1}Rk_*\cong R{k_c}_*i_c^{-1}.
\end{equation}

We then have the following sequence of natural isomorphisms:
\begin{equation}\label{id3}
\begin{split}
H^0\big({\widetilde{f}}^{-1}(c_0); {^p}\psi_{\widetilde{f}-c_0}(Rk_*\sF)\big) &\cong H^{-1}\big({\widetilde{f}}^{-1}(c); {{\widetilde i}_c}^{-1}Rk_*\sF\big) \\
& \overset{\eqref{bc}}{\cong} H^{-1}\big({\widetilde{f}}^{-1}(c); R{k_c}_*i_c^{-1}\sF\big) \\
& \cong H^{-1}\big(f^{-1}(c); i_c^{-1}\sF\big),
\end{split}
\end{equation}
where the first isomorphism follows from the properness of $\widetilde{f}$ and \cite[Example 10.4.20]{MS}.

To complete the proof of \eqref{id2n}, in view of \eqref{id3}  it suffices to show that there is a natural isomorphism
\begin{equation}\label{id3h}
H^0\big(\widetilde{f}^{-1}(c_0); {^p}\psi_{\widetilde{f}-c_0}(Rk_*\sF)\big) \cong H^{0}\big(f^{-1}(c_0); {^p}\psi_{f-c_0}(\sF)\big).
\end{equation}
In fact, if $k_{c_0}\colon f^{-1}(c_0) \hookrightarrow \widetilde{f}^{-1}(c_0)$ is the inclusion map induced by $k$, there is a natural isomorphism
\begin{equation}\label{id4h}
H^{0}\big(f^{-1}(c_0); {^p}\psi_{f-c_0}(\sF)\big) \cong H^0\big(\widetilde{f}^{-1}(c_0); R{k_{c_0}}_*{^p}\psi_{f-c_0}(\sF)\big). 
\end{equation}
Moreover, by \cite[Remark 4.3.7 (3)]{Sch}, there is a base change morphism 
\begin{equation}\label{bc3} {^p}\psi_{\widetilde{f}-c_0}\circ Rk_*\to R{k_{c_0}}_*\circ {^p}\psi_{f-c_0}\end{equation}
 inducing upon taking cohomology the natural map 
\begin{equation}\label{eq+3}
H^0\big(\widetilde{f}^{-1}(c_0); {^p}\psi_{\widetilde{f}-c_0}(Rk_*\sF)\big)  \to H^0\big(\widetilde{f}^{-1}(c_0); R{k_{c_0}}_*{^p}\psi_{f-c_0}(\sF)\big),
\end{equation}
which we will prove to be an isomorphism. 
For this, it thus suffices to show that the base change morphism \eqref{bc3} induces 
the following isomorphism in $D^b_c(\widetilde{f}^{-1}(c_0))$:
\begin{equation}\label{id4}
{^p}\psi_{\widetilde{f}-c_0}(Rk_*\sF) \cong R{k_{c_0}}_*{^p}\psi_{f-c_0}(\sF).
\end{equation}
After replacing $\sF$ by its Verdier dual, and using the fact that the perverse nearby cycle functor commutes with Verdier duality, it suffices to show that
\begin{equation}\label{id5}
{^p}\psi_{\widetilde{f}-c_0}(k_!\sF) \cong {k_{c_0}}_!{^p}\psi_{f-c_0}(\sF).
\end{equation}

Let $E:=X \setminus \C^n$ and let $\ell: E \hookrightarrow X$ denote the closed embedding. Since $f$ is a general linear function, it follows from the above Corollary \ref{cor33} that  the support of the (perverse) vanishing cycles of $\widetilde{f}$ is disjoint from $E$. Hence, we have for the induced closed inclusion $\ell_{c_0}:E \cap \widetilde{f}^{-1}(c_0) \hookrightarrow \widetilde{f}^{-1}(c_0)$ that
\begin{equation}\label{id6}
\ell_{c_0}^{-1}\big( {^p}\varphi_{\widetilde{f}-c_0}(k_!\sF)\big) = 0.
\end{equation}
Therefore, we have
\begin{equation}\label{id4b}
{^p}\varphi_{\widetilde{f}-c_0}(k_!\sF) \cong {k_{c_0}}_!k_{c_0}^{-1}\big({^p}\varphi_{\widetilde{f}-c_0}(k_!\sF)\big)\cong {k_{c_0}}_!\big({^p}\varphi_{f-c_0}(\sF)\big),
\end{equation}
where the first isomorphism follows from \eqref{id6} by using the attaching triangle for $(k_{c_0},\ell_{c_0})$, while the second isomorphism uses the fact that vanishing cycles commute with open embeddings (cf. also \cite[Proposition 10.4.19(2)]{MS}).

Consider next the distinguished triangles
\[
{k_{c_0}}_!\big({^p}\psi_{f-c_0}(\sF)\big)\to {k_{c_0}}_!\big({^p}\varphi_{f-c_0}(\sF)\big)\to {k_{c_0}}_! i_{c_0}^{-1}(\sF) \xrightarrow{+1}
\]
and
\[
{^p}\psi_{\widetilde{f}-c_0}(k_!\sF)\to {^p}\varphi_{\widetilde{f}-c_0}(k_!\sF)\to \widetilde{i}_{c_0}^{-1}(k_!\sF) \xrightarrow{+1},
\]
with $i_{c_0}:f^{-1}(c_0) \hookrightarrow \C^n$ and ${\widetilde i}_{c_0}:\widetilde{f}^{-1}(c_0) \hookrightarrow X$ the inclusion maps.
Note that there are natural base change morphisms connecting the terms of the first triangle to those of the second, respectively. 
By \eqref{id4b}, it follows  that to show \eqref{id5} it suffices to prove the isomorphism
\begin{equation}\label{id7}
{{\widetilde i}_{c_0}}^{-1}k_!\sF\cong {k_{c_0}}_!i_{c_0}^{-1}\sF.
\end{equation}
This is immediate, either by using directly the base change formula (e.g., see \cite[Chapter V, Proposition 10.7]{Bo}), or by noting that, if ${\widetilde i}^E_{c_0}: E \cap \widetilde{f}^{-1}(c_0) \hookrightarrow E$ denotes  the inclusion map, then
\[\ell_{c_0}^{-1} {{\widetilde i}_{c_0}}^{-1}k_!\sF = ({\widetilde i}^E_{c_0})^{-1}\ell^{-1} k_!\sF=0,\]
whence, using the attaching triangle for the pair  $(k_{c_0},\ell_{c_0})$ yields 
\[{{\widetilde i}_{c_0}}^{-1}k_!\sF\cong  {k_{c_0}}_! k_{c_0}^{-1} {{\widetilde i}_{c_0}}^{-1}k_!\sF \cong {k_{c_0}}_! i_{c_0}^{-1} k^{-1} k_!\sF\cong {k_{c_0}}_!i_{c_0}^{-1}\sF.\]
This completes the proof of \eqref{id4} and of the Proposition.
\end{proof}

\end{document}